\documentclass[11pt]{article}
\usepackage{latexsym, amsmath, amssymb,color}

\usepackage{amsmath}
\usepackage{amsfonts}
\usepackage{amssymb}
\usepackage{amscd}
\usepackage{amsthm}
\usepackage{fancyhdr}
\usepackage{color}
\theoremstyle{plain}
\newtheorem{theorem}{Theorem}[section]
\newtheorem{proposition}[theorem]{Proposition}
\newtheorem{remark}[theorem]{Remark}
\newtheorem{lemma}[theorem]{Lemma}
\newtheorem{definition}[theorem]{Definition}
\newtheorem{corollary}[theorem]{Corollary}

\newcommand\Ee{{\mathcal E}}

\newcommand\sO{{\mathcal O}}
\newcommand\Oo{{\mathcal O}}

\newcommand\PP{{\mathbf P}}

  \def \tab#1{\kern #1 truein}

\begin{document}
\title{Castelnuovo-Mumford Regularity and Splitting Criteria for Logarithmic Bundles over Rational Normal Scroll Surfaces}
\author{R. Di Gennaro and F. Malaspina 
\vspace{6pt}\\
{\small   Universit\'a  degli Studi di Napoli Federico II}\\
 {\small\it Complesso Universitario Monte Sant'Angelo, Via Cinthia, 80126 Napoli, Italy}\\
{\small\it e-mail: digennar@unina.it}
\vspace{6pt}\\
 {\small  Politecnico di Torino}\\
{\small\it  Corso Duca degli Abruzzi 24, 10129 Torino, Italy}\\
{\small\it e-mail: francesco.malaspina@polito.it}}
\maketitle 
\def\thefootnote{}

\begin{abstract}
We introduce and study a notion of Castelnuovo-Mumford regularity
suitable for rational normal scroll surfaces. In this setting we prove analogs of some classical properties.  We prove splitting criteria for coherent sheaves and a characterization of Ulrich bundles. Finally we study logarithmic bundles associated to arrangements of lines and rational curves.
\end{abstract}
\section{Introduction}\label{sect:intro}
In chapter~14 of~\cite{m} Mumford introduced the concept of regularity
for a coherent sheaf on a projective space $\PP^n$. It was soon clear
that it was a key notion and a fundamental tool in many areas of
algebraic geometry and commutative algebra.

From the algebraic geometry point of view, regularity measures the
complexity of a sheaf: the regularity of a coherent sheaf is an
integer that estimates the smallest twist for which the sheaf is
generated by its global sections. In Castelnuovo's much earlier
version, if $X$ is a closed subvariety of projective space and $H$ is
a general hyperplane, one uses linear systems (seen now as a precursor
of sheaf cohomology) to get information about $X$ from information
about the intersection of $X$ with $H$ plus other geometrical or
numerical assumptions on $X$.

From the computational and commutative algebra point of view, the
regularity is one of the most important invariants of a finitely
generated graded module over a polynomial ring. Roughly, it measures
the amount of computational resources that working with that module
requires. More precisely the regularity of a module bounds the largest
degree of the minimal generators and the degree of syzygies.

Extensions of this notion have been proposed over the years to handle
other ambient varieties instead of projective space: Grassmannians \cite{am},
quadrics \cite{bm3}, multiprojective spaces \cite{hw,bm2,cm2}, $n$-dimensional smooth
projective varieties with an $n$-block collection~\cite{cm2},
and abelian varieties \cite{PP}.
For a different approach to multigraded
regularity from a commutative algebra point of view, see \cite{bc,be}.

The aim of this paper is to introduce  a very simple and natural concept of regularity on a rational normal scroll surface.
  
The interesting fact is that on $\mathcal Q_2\cong \mathbb P^1\times \mathbb P^1$ our definition of regularity coincides
with
this definition of  regularity on $\mathbb P^1\times \mathbb P^1$ given in \cite{hw,cm2,bm3,bm2}
and we are able to prove that every regular coherent sheaf is globally generated, as done by Mumford
in the classical case $\mathbb {P}^n$.

  The second aim of this paper is to apply our notion of regularity in order to investigate under what circumstances a vector bundle can be decomposed into a direct sum of line bundles.
  In particular, in the second section splitting criteria of vector bundle on a rational normal scroll surface are given, generalizing some analogous result already known for $\PP^1\times \PP^1$ (\cite{bm2},\cite{bm3}). In \cite{FuMa} the authors give some splitting criteria for vector bundles of rank 2 in terms of Chern classes and vanishing of certain cohomology groups using Beilinson type spectral sequence.  They also remark that their results are the best possible without analysing the differentials in the spectral sequence. Our splitting criteria work for vector bundles of arbitrary rank thanks to the use of regularity and without the use of spectral sequences.
  In \cite{FM} Theorem B is given a complete classification of Ulrich bundles on rational normal scroll surfaces. Here we give an alternative and simpler proof without using derived category techniques.
  
  Finally, the last section focuses on the logarithmic bundle of divisors on a rational normal scroll. It fits in the  classical topic of the study of normal crossing divisors on a smooth complex variety $X$. When $D$ is a
normal crossing divisor, Deligne \cite{D}  constructed a mixed Hodge
structure on $U = X \setminus D$ using the logarithmic de Rham complex
$\Omega_X^\bullet(-\mbox{log }D)$. Following this idea, in \cite{S} 
Saito defined the sheaf $T_X(D)$ of derivations tangent to $D$ and (dually) the sheaf of logarithmic 
one-forms with pole along $D$, the logarithmic bundle ${\Omega^1}_X(\log D)$.\\
The module of tangent derivations is a sheaf of
$\mathcal{O}_X$--modules, such that if $f \in \mathcal{O}_{X,p}$ is a local
defining equation for $D$ at $p$, then 
\[
(T_X(-\log D))_p = \{ \theta \in T_X| \theta(f) \in \langle f \rangle \}.
\]
\noindent When $D$ is a normal-crossing divisor, $T_X(-\log D)$ is always locally free. 

%
The module of derivations tangent to $D$ is a reflexive sheaf. So, since a reflexive sheaf on a surface is always locally free, so it is interesting understand when it (or its dual) splits as $O_{\PP ^2}(a) \oplus O_{\PP ^2}(b)$, in which case the divisor is said to be a {\em free}. In general, free divisors are difficult to find. We find some classes of free divisors (precisely free arrangements of lines and rational curves) on a two dimensional rational normal scroll.\\

We thank M. Aprodu, A.P. Rao, G. Casnati and J. Pons-Llopis for helpful discussions and fundamental observations.

   \section{Regularity  on $S(a_0,a_1)$}

Throughout this article, our base field  is algebraically closed with characteristic 0.   Let $X=S(a_0, a_1)$ be a smooth rational normal scroll, the image of $\PP (\Ee)$ via the morphism defined by $\Oo_{\PP (\Ee)}(1)$, where $\Ee \cong  \Oo_{\PP^1}(a_0)\oplus\Oo_{\PP^1}(a_1)$ is a vector bundle of rank $n+1$ on $\PP^1$ with $0< a_0 \le a_1$. Letting $\pi : \PP (\Ee) \rightarrow \PP^1$ be the projection, we may denote by $H$ and $f$, the hyperplane section corresponding to $\Oo_{\PP(\Ee)}(1)$ and the fibre corresponding to $\pi^*\Oo_{\PP^1}(1)$, respectively. Then we have $Pic (X)\cong \mathbf Z\langle H,f\rangle$ and $\omega_X \cong \Oo_X(-2H+(c-2)f)$, where $c:=a_0+a_1$ is the degree of $X$. 

For the computational purpose, we use the following lemma.
\begin{lemma}[\cite{EH}]\label{lem}
For any $i=0,1,2$, we have
\begin{itemize}
\item [(i)] $H^i(X, \Oo_X(aH+bf)) \cong H^i(\PP^1, \mathrm{Sym}^a\Ee \otimes \Oo_{\PP^1}(b))$ if $a\ge 0$;
\item [(ii)]  $H^i(X, \Oo_X(-H+bf)) =0$ for any $b$;
\item [(iii)] $H^i(X, \Oo_X(aH+bf)) \cong H^{2-i}(\PP^1, \mathrm{Sym}^{-a-2}\Ee \otimes \Oo_{\PP^1}(c-b-2))$ if $a\le -2$.
\end{itemize}
\end{lemma}

Recall the dual of the relative Euler exact sequence of $X$:

\begin{equation}\label{eq1b}
0\to \Oo_X(-H+cf)\to \Oo_X(a_0f)\oplus\Oo_X(a_1f) \to \Oo_X(H) \to 0.
\end{equation}
The pullback of the Euler sequence in $\PP^1$ is
\begin{equation}\label{eq2b}
0\to \Oo_X(-f)\to \Oo^2_X\to \Oo_X(f) \to 0,
\end{equation}
and we obtain

\begin{equation}\label{eq3b}
0\to \Oo_X(-H+(c-2)f)\to\Oo^2_X(-H+(c-1)f)\to \Oo_X(a_0f)\oplus\Oo_X(a_1f) \to \Oo_X(H) \to 0,
\end{equation}

 We give a definition of regularity on $X$:
\begin{definition}\label{d1}
A  coherent sheaf $F$ on $X$  is said to be  {\it $(p,p')$-regular} if, denoting $E=F(pH+p'f)$, for all $i>0$, $$h^2(E(-H+(c-2)f))=   h^1(E(-H+(c-1)f))= h^1(E(-f))=0$$
\\
We will say {\it regular} in order to $(0,0)$-regular.\\
We will say {\it $p$-regular} in order to $(p,0)$-regular.\\
We define the {\it regularity} of $F$, $Reg (F)$, as the least integer $p$ such that $F$ is $p$-regular. We set
$Reg (F)=-\infty$ if there is no such integer.
\end{definition}
\begin{remark} When $c=2$ we get $X=\PP^1\times\PP^1$ and this notion of regularity coincides with the notions of Castelnuovo-Mumford regularity given in \cite{bm2}, \cite{bm3} and \cite{hw} since $h^2(E(-H+(c-2)f))=h^2(E(-1,-1)),   h^1(E(-H+(c-1)f))=h^1(E(-1,0)), h^1(E(-f))=h^1(E(0,-1))$.
\end{remark}

\begin{lemma}\label{l1} If $F$ is a regular coherent sheaf on $X$,
then $h^1(F_{|f}((a-1)H+bf))=0$ for any $a\geq 0$ and for any integer $b$.

\end{lemma}
\begin{proof} Let us consider this exact cohomology sequence:
$$\dots \rightarrow H^1(F(-H+(c-1)f)) \rightarrow H^1(F_{|f}(-H+(c-1)f)) \rightarrow H^{2}(F(-H+(c-2)f))\rightarrow \dots$$
Since the first and the
third groups vanish  by hypothesis, then also the middle group vanishes. $H^1(\Oo_{|f}(-H+(c-1)f))\cong H^1(\PP^1,\Oo_{\PP^1}(-1))$ so $H^1(F_{|f}((a-1)H+bf))=0$ for any $a\geq 0$ and for any integer $b$.\\

\end{proof}

\begin{lemma}\label{l2} If $F$ is a regular coherent sheaf on $X$,
then $h^2(F((a-1)H+(c-2+b)f))=0$ for any $a,b\geq 0$ and $h^1(F(tf)=0$ for any $t\geq -1$.
\end{lemma}
\begin{proof}From (\ref{eq2b}) we get $h^2(F(-H+(c-2+t)f))=0$ for any $t\geq 0$. From (\ref{eq1b}) tensored by $F((c-2)f)$ we get $h^2(F((c-2+t)f))=0$ and again by (\ref{eq2b}) we obtain $h^2(F((c-2+t)f))=0$ for $t\geq 0$. In the same way $h^2(F((a-1)H+(c-2+b)f))=0$ for any $a\geq 0$ and for any $b\geq 0$.\\
From $$0\to F(-f)\to F\to F_{|f}\to 0,$$ we deduce that
$h^1(F(tf))=0$ for any $t\geq -1$.\\

\end{proof}

\begin{lemma}\label{l3} If $F$ is a regular coherent sheaf on $X$,\begin{enumerate}
    \item [i)] $H^1(F_{|H}((c-1+b)f))=0$ for any $b\geq 0$.\\
\item [ii)] $H^1(F_{|H}((a+1)H+(b-1)f)=0$ for any $a,b\geq 0$.
\end{enumerate}

\end{lemma}
\begin{proof}Let us consider this exact cohomology sequence:
$$\dots \rightarrow H^1(F((c-1)f)) \rightarrow H^1(F_{|H}((c-1)f)) \rightarrow H^{2}(F(-H+(c-2)f))\rightarrow \dots$$
Since $H^1(F((c-1)f))=H^{2}(F(-H+(c-2)f)=0$ we get $h^1(F_{|H}((c-1)f))=h^1(\PP^1, F_{H}(c-1))=0$ and also $h^1(\PP^1, F_{|H}(c-1+t))=h^1(F_{|H}((c-1+t)f))=0$ for $t\geq 0$. So $(i)$ is proved.\\
$H^1(\Oo_{|H}((a+1)H+(b-1)f)\cong H^1(\PP^1,\Oo_{\PP^1}((a+1)c+b-1))$ and when $a,b\geq 0$ we get $(a+1)c+b-1\geq c-1$ so $H^1(F_{|H}((a+1)H+(b-1)f)=0$ for any $a,b\geq 0$. So also $(ii)$ is proved.

\end{proof}

\begin{proposition}\label{p1} Let $F$ be a regular coherent sheaf on $X$ then
  \begin{enumerate}
  \item $F(pH+p'f)$ is regular for $p,p'\geq 0$.\\
  \item $H^0(F(f))$ is spanned by $$H^0(F)\otimes H^0(\sO(f));$$ and $H^0(F(H))$ it is
spanned by
  $$H^0(F(a_0f))\oplus H^0(F(a_1f)).$$
  \end{enumerate}
  \end{proposition}

  \begin{proof} $(1)$ Let $F$ be a regular coherent sheaf, we want show that also $F(H)$ is regular. $h^2(F((c-2)f))=0$ by Lemma \ref{l2}. In order to show that  $h^1(F((c-1)f))=0$ let us
consider the exact
  cohomology sequence:
$$\dots \rightarrow H^1(F(-H+(c-1)f)) \rightarrow H^1(F((c-1)f)) \rightarrow H^{1}(F_{|H}((c-1)f))\rightarrow \dots$$
We notice that the first group vanishes by hypothesis and the third group vanishes by $(i)$ of Lemma \ref{l3}.
Then also the middle group vanishes.\\
It remains to show that $h^1(F(H-f))=0$. By the exact
   cohomology sequence:
$$\dots \rightarrow H^1(F(-f)) \rightarrow H^1(F(H-f)) \rightarrow H^{1}(F_{|H}(H-f))\rightarrow \dots$$
since the first group vanishes by hypothesis and the third group vanishes by $(ii)$ of Lemma \ref{l3} we obtain that also the middle group vanishes.\\

Let $F$ be a regular coherent sheaf, we want show that also $F(f)$ is regular. $h^2(F(-H+(c-1)f))=h^1(F)=0$ by Lemma \ref{l2}. In order to show that  $h^1(F(-H+cf))=0$ let us
consider the exact
   cohomology sequence:
$$\dots \rightarrow H^1(F(-H+(c-1)f)) \rightarrow H^1(F(-H+cf)) \rightarrow H^{1}(F_{|f}(-H+cf))\rightarrow \dots$$
We notice that the first group vanishes by hypothesis and the third group vanishes by Lemma \ref{l1}.
Then also the middle group vanishes.\\

$(2)$ Let us consider (\ref{eq2b}) tensored by $F$:

$$0\to F(-f)\to F^2\to F(f) \to 0.$$
Since $H^1(F(-f))=0$, we obtain
$$ H^0(F)\otimes H^0(\sO_X(f))\to H^0(F(f)) \to 0.$$
Now let us consider (\ref{eq3b}) tensored by $F$:

$$0\to F(-H+(c-2)f)\to F^2(-H+(c-1)f)\to F(a_0f)\oplus F(a_1f) \to F(H) \to 0.$$
Since $H^2(F(-H+(c-2)f))= H^1(F(-H+(c-1)f)=0$, we obtain
$$ H^0(F(a_0f))\oplus H^0(F(a_1f)) \to H^0(F(H)) \to 0.$$
\end{proof}

  \begin{remark}\label{gg} If  $F$ is a regular coherent sheaf on $X$ then it is globally generated.\\
  In fact by the above proposition we have the following surjections:
  $$H^0(F)^q\to
H^0(F(a_0f))\oplus H^0(F(a_1f)) \to H^0(F(H)),$$ for a suitable positive integer $q$. So also
the map
$$H^0(F)^q\rightarrow H^0(F(H))$$ is a
  surjection.\\
  Moreover we can consider a sufficiently large twist $l$ such that $F(lH)$ is globally generated. For a suitable positive integer $q'$ the
  commutativity of the diagram
$$\begin{matrix}
H^0(F)^{q'}\otimes\sO_X
&\to&H^0(F(lH))\otimes\sO_X\\
\downarrow&&\downarrow\\
H^0(F)^q\otimes\sO(lH)&\to&F(lH)
\end{matrix}$$
and the surjectivity of the top horizontal map and the two vertical maps yield the surjectivity of $H^0(F)\otimes\sO(lH)\to F(lH)$, which implies that $F$ is generated by its
sections.
  \end{remark}
  \begin{remark}\label{lb}
  $\sO_X(aH-bf) $  is regular if and only if $a\geq 0$ and $b\leq aa_0$.
   \end{remark}
  \begin{remark}\label{b1} 
  In particular $\sO_X, \sO_X(f), \sO_X(H-f) $ are regular but not $-1$-regular so $Reg(\sO_X)=Reg(\sO_X(f))=Reg(\sO_X(H-f))=0$.\\
   \end{remark}
 
 \section{Splitting criteria and Ulrich bundles}
 It is possible to use this notion of regularity in order to prove splitting criteria for vector bundles:
  \begin{theorem}\label{t1}Let $E$ be a rank $r$ vector bundle on $X$.\\ Then following conditions are equivalent:
  \begin{enumerate}
  \item for any integer $t$,  $$h^1(E(tH+(c-1)f))= h^1(E(tH-f))=0,$$ 
  \item There are $r$ integer $t_1, \dots, t_r$ such that $E\cong \bigoplus_{i=1}^r \sO_X(t_iH)$.
  \end{enumerate}
  \end{theorem}

  \begin{proof} $(1)\Rightarrow (2)$. Let assume that $t$ is an integer such that $E(tH)$ is regular but
  $E((t-1)H)$ not.\\
By the definition of regularity and $(1)$ we can say that $E((t-1)H)$ is not regular if and only if
$H^{2}(E((t-2)H+(c-2)f))\not=0$. By Serre duality we have that $H^0(E^{\vee}(-tH))\not=0$.\\
 Now since $E(tH)$ is globally generated by Remark \ref{gg} and $H^0(E^{\vee}(-tH))\not=0$ we can conclude
 that $\sO_X$ is a direct summand of $E(tH)$.
      By iterating these arguments we get $(2)$.\\
      $(2)\Rightarrow (1)$. $h^1(\sO_X(tH+(c-1)f))= h^1(\sO_X(tH-f))=0,$ for any integer $t$, so if $E\cong \bigoplus_{i=1}^r \sO(t_iH)$ then it satisfies all the conditions in $(1)$.\\
      \end{proof}
      \begin{remark}If $c=2$ the above theorem is the Horrocks criterion on $\PP^1\times\PP^1$ (see \cite{bm2}, \cite{bm3}).\\
\end{remark}
\begin{corollary}\label{cor1} Let $E$ be a vector bundle on $X$ with $Reg(E)=0$ and $H^2(E(-2H+(c-2)f))\not=0$ or $H^1(E(-2H+(c-1)f))=H^1(E(-H-f))=0$, then $\sO_X$ is direct summand of $E$.
\end{corollary}
\begin{proof}
    Since $E(-H)$ is not regular, if $H^1(E(-2H+(c-1)f))=H^1(E(-H-f))=0$, we have $H^2(E(-2H+(c-2)f))\not=0$ and $E\cong \sO_X$ by the proof of the above Theorem.
\end{proof}

 \begin{theorem}\label{t2}Let $E$ be a  vector bundle on $X$.\\ Then following conditions are equivalent:
  \begin{enumerate}
  \item  for any integer $t$,  $$H^1(E(tH))=H^1(E(tH+(c-2)f))=H^1(E(tH+(a_0-1)f))=H^1(E(tH+(a_1-1)f))=0$$
  
  \item $E$ is a direct sum of  line bundles $\sO_X$, $\sO(f)$ and $\sO(H-f)$ with some twist $tH$.
  \end{enumerate}
  \end{theorem}\begin{proof}$(1)\Rightarrow (2)$. 
  Let assume that $t$ is an integer such that $E(tH)$ is regular but
  $E((t-1)H)$ not. Up to a twist we may assume $t=0$.\\
By the definition of regularity and $(1)$ we can say that $E(-H)$ is not regular if and only if one of the
following conditions is satisfied:
     \begin{enumerate}
     \item[i] $h^2(E(-2H+(c-2)f))\not=0$,
      \item[ii] $h^1(E(-2H+(c-1)f))\not=0$.
   \item[iii] $h^1(E(-H-f))\not=0$. \end{enumerate}
  Let us consider one by one the conditions:\\
     $(i)$ Let $h^2(E(-2H+(c-2)f))\not=0$, we can conclude that $\sO_X$ is a direct
summand as in the above theorem.\\
$(ii)$ Let $h^1(E(-2H+(c-1)f))\not=0$. Let us consider the exact sequence:

      $$0 \to E(-2H+(c-1)f)\to E(-H+(a_0-1)f)\oplus E(-H+(a_0\textcolor{red}{a_1?}-1)f)  \rightarrow E(-f)\rightarrow 0$$
       
Since $$ H^1(E(-H+(a_0-1)f)=H^1(E(-H+(a_1-1)f)=0,$$
we have a surjective map
 $$H^0(E(-f))\to H^1(E(-2H+(c-1)f)).$$ Therefore $H^0(E\otimes\sO_X(-f))\not=0$ and there exists a non zero map$$  g: \sO_X(f) \rightarrow  E.$$.\\
On the other hand  $$H^1(E(-2H+(c-1)f))\cong H^1(E^{\vee}(-f))$$ so let us consider the exact sequence
 $$0 \to 
  E^{\vee}(-f) \rightarrow 2E^\vee\to E^{\vee}(f)
\to 0.$$
Since
$$H^1( E^\vee)=H^1(E(-2H+(c-2)f))=0,$$

 we have a surjective map  $$H^0(E^{\vee}(f))\to H^{1}(E(-2H+(c-1)f)).$$ Therefore
 $H^0(E^{\vee}\otimes \sO_X(f))\not=0$ and there exists a non zero map$$  h: E\rightarrow  \sO_X(f).$$
 Let us consider the following commutative diagram:
 $$\begin{array}{ccc} H^{1}(E(-2H+(c-1)f))\otimes H^1(E^{\vee}(-f))& \xrightarrow{\sigma}& H^{2}(E(-2H+(c-2)f))\cong\mathbb C\\ \downarrow & &\downarrow \\
 H^0(E(-f)) \otimes H^1(E^{\vee}(-f))& \xrightarrow{\mu}& H^{1}(\sO_X(-f)\otimes \sO_X(-f))\cong\mathbb C\\
 \downarrow & &\downarrow \\
      H^0(E(-f))  \otimes H^0(E^{\vee}(f))& \xrightarrow{\tau}& H^{0}(\sO_X(-f)\otimes \sO_X(f))\cong\mathbb C\\
    \uparrow{\cong}&&\uparrow{\cong}\\
     {Hom}(E, \sO_X(f))\otimes{Hom}(\sO_X(f),E)&\xrightarrow{\gamma}&{Hom}(\sO_X(f),\sO_X(f)).    \end{array}$$The map $\sigma$ comes from Serre duality and it is not zero, the right vertical map are isomorphisms and the left vertical map are surjective so also the map $\tau$ is not zero.\\  This means that the the map $$h\circ g : \sO(f) \to \sO(f)$$ is non-zero and hence it is an
isomorphism.\\
This isomorphism shows that $\sO(f)$ is a direct summand of $E$.\\
$(iii)$ Let $h^1(E(-H-f))\not=0$. Let us consider the exact sequence:

      $$0 \to E(-H-f)\to E(-H) \rightarrow E(-H+f)\rightarrow 0.$$ Since $h^1(E(-H))=0$ we get $h^0(E(-H+f))\not=0$. By Serre duality $h^1(E(-H-f))=h^1(E^\vee(-H+(c-1)f))$. By the exact sequence
      $$0\to E^\vee(-H+(c-1)f)\to E^\vee((a_0-1)f)\oplus E^\vee((a_1-1)f)  \rightarrow E^\vee(H-f)\to 0.$$
      Since $h^1(E^\vee((a_0-1)f))=h^1(E(-2H+(a_1-1)f))=0$ and $h^1(E^\vee((a_1-1)f))=h^1(E(-2H+(a_0-1)f))=0$ we get also $h^0(E^\vee(H-f))\not=0$. By arguing as above we can conclude that $\Oo_X(H-f))$ is a direct summand of $E$.\\
  $(2)\Rightarrow (1)$. As in Theorem \ref{t1}.

\end{proof}
\begin{remark}If $c=2$ we get $a_0=a_1=1$ so $H^1(E(tH+(c-2)f))=H^1(E(tH+(a_0-1)f))=H^1(E(tH+(a_1-1)f))$ coincide with $H^1(E(tH))$ and we have exactly the classification of the ACM bundles on $\PP^1\times\PP^1$ (see \cite{Kn}).\\
The proof in this case coincides with \cite{bm3} Theorem $1.4$.
\end{remark}
\begin{remark}If $c=3$ we get $a_0=1, a_1=2$ so $H^1(E(tH+(c-2)f))=H^1(E(tH+(a_1-1)f))=H^1(E(tH+f))$ and $H^1(E(tH+(a_0-1)f))=H^1(E(tH))$. It is well known (see \cite{EH2}) that the only indecomposable ACM bundles on $X=S(1,2)$ are $\Oo_X,\Oo_X(f), \Oo_X(2f), \Oo_X(H-f)$ and the rank two vector bundle obtained as the extension among $\Oo_X(H-f)$ and $\Oo_X(2f)$. Since the condition $H^1(E(tH+f))=0$ for any integer $t$ is not satisfied by $\Oo_X(2f)$ we obtain again Theorem \ref{t2}.

\end{remark}

\begin{remark}If $c=4$ we get $a_0=2, a_1=2$ or $a_0=1, a_1=3$ so in the fist case $H^1(E(tH+(c-2)f))=H^1(E(tH+2f))$ and $H^1(E(tH+(a_0-1)f))=H^1(E(tH+(a_1-1)f))=H^1(E(tH+f))$ and in the second case $H^1(E(tH+(c-2)f))=H^1(E(tH+(a_1-1)f))=H^1(E(tH+2f))$ and $H^1(E(tH+(a_0-1)f))=H^1(E(tH))$. It is well known (see \cite{FM})  the classification of indecomposable ACM bundles on $X=S(2,2)$ or $X=S(1,3)$ and it is easy to  obtain again Theorem \ref{t2}.

\end{remark}
\begin{corollary}\label{cor2} Let $E$ be an indecomposable vector bundle on $X$ with $Reg(E)=0$ and $H^2(E(-2H+(c-2)f))=0$.

\begin{enumerate}
     
      \item  If $h^1(E(-2H+(c-1)f))\not=0$ and $H^1(E(-H))=H^1(E(-2H+(a_0-1)f))=H^1(E(-2H+(a_1-1)f))=0$ then $E\cong\sO_X(H-f)$.
   \item $h^1(E(-H-f))\not=0$ and $H^1(E(-2H+(c-2)f))=H^1(E(-H+(a_0-1)f))=H^1(E(-H+(a_1-1)f))=0$ then $E\cong\sO_X(f)$. \end{enumerate}

\end{corollary}
\begin{proof}
    Since $E(-H)$ is not regular and $H^2(E(-2H+(c-2)f))=0$, if $H^1(E(-2H+(c-1)f))=0$ then $H^1(E(-H-f))\not=0$ and viceversa. So, thanks the other vanishings, by the proof of the above Theorem, we obtain $(i)$ and $(ii)$.
\end{proof}

For $c>4$ the family of ACM bundles are too complicated (see \cite{FM}) but we can use our notion of regularity to study Ulrich bundles. We need the following Lemmas: 

\begin{lemma}\label{l5b}
If $E$ is a globally generated  ACM bundle on $X$, then 
 $$H^1(E((a-1)H+bf))=H^1(E(aH+(b-1)f))=0$$ for any $a,b\geq 0$.

\end{lemma}
\begin{proof} Since $E$ is globally generated we have a surjective map $$\Oo_X\to E\to 0.$$
 Since $h^2(\Oo_X(-H-f))=0$ we obtain $h^2(E(-H-f))=0$.\\
Let us consider this exact cohomology sequence:
$$\dots \rightarrow H^1(E(-H)) \rightarrow H^1(E_{|f}(-H)) \rightarrow H^{2}(E(-H-f))\rightarrow \dots$$
Since the first and the
third groups vanish  by hypothesis, then also the middle group vanishes. As in Lemma \ref{l2} $H^1(E_{|f}((a-1)H+bf))=0$ for any $a\geq 0$ and for any integer $b$. This implies $H^1(E((a-1)H+bf))=0$ for any $a,b\geq 0$.\\
By sequence (\ref{eq3b}) tensored by $E(-H-2f)$ we get $H^2(E(-2f))=0$. Let us consider this exact cohomology sequence:
$$\dots \rightarrow H^1(E) \rightarrow H^1(E(-f)) \rightarrow H^{2}(E(-f))\rightarrow \dots$$
Since the first and the
third groups vanish  by hypothesis, then also the middle group vanishes. We obtain
$H^1(E((aH+(b-1)f))=0$ for any $a,b\geq 0$.\\

\end{proof}
 
\begin{lemma}\label{l5}
If $E$ is an Ulrich bundle on $X$, then \begin{enumerate}
\item [i)] $H^2(E((a-2)H+bf))=0$ for any $a,b\geq 0$.
\item [ii)] $H^1(E((a-1)H+bf))=H^1(E(aH+(b-1)f))=0$ for any $a,b\geq 0$.
\item [iii)] $E$ is regular.
\end{enumerate}
\end{lemma}
\begin{proof} Since $E$ is Ulrich we have
$$h^i(E(-H))=h^i(E(-2H))=0$$ for any $i$. So we obtain $(i)$ as in Lemma \ref{l1}.\\
Since an Ulrich bundle is ACM and globally generated, by the abve Lemma $(ii)$ is proved.\\
By $(i)$ and $(ii)$ we obtain the vanishing of Definition \ref{d1} and hence $(iii)$.

\end{proof}

Thanks to our notion of regularity and the above we can give a simpler proof of \cite{FM} Theorem B without a Beilinson type spectral sequence:
\begin{theorem}[\cite{FM} Theorem B]\label{Ul0}
 An indecomposable $E$ on $X$ is Ulrich if and only if $E$ fits into:
\begin{equation}\label{Ul1}
 0 \to \sO_X(H-F)^{a} \to E \to \sO_X((c-1)f)^{b} \to 0, \qquad \mbox{for some
   $a,b \ge 0$.}
\end{equation}
\end{theorem}
\begin{proof}
By $(iii)$ of Lemma \ref{l5}, $E$ is regular and since $h^0(E(-H))=0$, $E(-H)$ must be not regular.  
  By $(i)$ and $(ii)$ of Lemma \ref{l5}, $h^2(E(-2H+(c-2)f))=0$ and we may conclude that one of the
following conditions is satisfied:
     \begin{enumerate}
      \item [$\alpha$)] $h^1(E(-H-f))\not=0$.
   \item[$\beta)$] $h^1(E(-2H+(c-1)f))\not=0$. \end{enumerate}
   $\alpha)$ Let $h^1(E(-H-f))=a\not=0.$
  Let us consider the exact sequence:

      $$0 \to E(-H-f)\to E(-H) \rightarrow E(-H+f)\rightarrow 0.$$ Since $h^1(E(-H))=h^0(E(-H))=0$ we get $h^0(E(-H+f))=a$.\\
      So there exists a map $$  h: \sO(H-f)^a\rightarrow  E.$$ 
      Let $b=h^1(E(-2H+f))$. We distinguish two cases: $b=0$ and $b\not=0$.\\
      Let assume first $b=0$.
      By Serre duality $h^1(E(-H-f))=h^1(E^\vee(-H+(c-1)f))=a$. 
      From
      $$\dots \rightarrow H^1(E(-2H+f)) \rightarrow H^1(E_{|f}(-2H+f)) \rightarrow H^{2}(E(-H))\rightarrow \dots$$
since the first ($b=0$) and the
third groups vanish  by hypothesis, then also the middle group vanishes. As in Lemma \ref{l2} $H^1(E_{|f}((a-2)H+bf))=0$ for any $a\geq 0$ and for any integer $b$. This implies $H^1(E((a-1)H+bf))=0$ for any $a,b\geq 0$. In particular $h^1(E^\vee((a_0-1)f))=h^1(E(-2H+(a_1-1)f))=0$ and $h^1(E^\vee((a_1-1)f))=h^1(E(-2H+(a_0-1)f))=0$ so, from
 the exact sequence
      $$0\to E^\vee(-H+(c-1)f)\to E^\vee((a_0-1)f)\oplus E^\vee((a_1-1)f)  \rightarrow E^\vee(H-f)\to 0,$$ we get also $h^0(E^\vee(H-f))\not=0$. Hence as in Theorem \ref{t2} we obtain $E\cong\Oo_X(H-f)$.\\
Let assume now $b>0$. By \cite{CH} we may assume that the kernel $K$ and the cokernel $G$ of $h$ are also Ulrich. So we obtain two exact sequences with also $U$ Ulrich:
$$0\to U\to E\to G\to 0,$$ and
\begin{equation}\label{kk}0\to K\to\Oo_X(H-f)^a\to U\to 0.\end{equation}
Notice that by Lemma \ref{l5} and sequence (\ref{eq1b}) tensored by $U(-2H-f)$ we get $h^2(U(-H-f))=0$. So
if we twist the sequence (\ref{kk}) by $-H-f$, since $h^0(U(-H-f))=h^2(U(-H-f))=0$ and $h^1(\Oo_X(-2f)^a))=a$, we have $h^1(K(-H-f))\not= 0$. Now let us consider the sequence in cohomology:
$$H^0(U(-2H+f))\to H^1(K(-2H+f))\to H^1(\Oo_X(-H)^a)$$
Since $H^0(U(-2H+f))=H^1(\Oo_X(-H)^a)=0$, we get $H^1(K(-2H+f))=0$ and by the argument above for $b=0$ we may conclude that $\Oo_X(H-f)$ is a direct summand of $K$. By iterating this argument we get the sequence (\ref{kk}) simply becomes
$$0\to\Oo_X(H-f)^{a'}\to\Oo_X(H-f)^a\to\Oo_X(H-f)^{a-a'}\to 0$$ for a suitable positive integer $a'$. Hence we may assume that $h$ is injective.
Let us denote by $G$ the cokernel of $h$:
$$0\to\Oo_X(H-f)^a\to E\to G\to 0.$$
Notice that $h^1(G(-H-f))=0$.
If we twist the above exact sequence by $\Oo_X(-2H+tf)$, since $h^i(\Oo_X(-H+tf))=0$ for any $i,t$ we get $b=h^1(G(-2H+f))$ and $h^1(G(-2H+tf))=h^1(E(-2H+tf))$ for any integer $t$.\\
From $$0 \to 
  G(-H-2f) \rightarrow G(-H-f)^2\to G(-H)
\to 0,$$ we obtain $h^1(G(-H-tf))=0$ for any $t\geq 0$.\\
Let us consider the exact sequence:

      $$0 \to G(-2H+f)\to G(-H+(-a_0+1)f)\oplus G(-H+(-a_1+1)f)  \rightarrow G(-(c-1)f)\rightarrow 0$$
       
Since $$ H^i(G(-H+(-a_0+1)f))=H^i(G(-H+(-a_1+1)f))=0,$$ for any $i$,
we have $H^0(G\otimes\sO_X(-(c-1)f))=b\not=0$.
On the other hand  $$H^1(G(-2H+f))\cong H^1(G^{\vee}((c-3)f))$$ so let us consider the exact sequence
 $$0 \to 
  G^{\vee}((c-3)f) \rightarrow 2G^\vee((c-2)f)\to G^{\vee}((c-1)f)
\to 0.$$
Since
$$H^1( G^\vee((c-2)f))=H^1(G(-2H))=0,$$

 we get
 $H^0(G^{\vee}\otimes \sO_X((c-1)f))=b\not=0$. So by arguing as in Theorem \ref{t2} we obtain $G\cong \sO_X((c-1)f)^b$ and $E$ fits in (\ref{Ul1}).\\
$ \beta)$ Notice that if $E$ is Ulrich also $E'=E^\vee (H+(c-2)f)$ is Ulrich. The condition $$h^1(E(-2H+(c-1)f))=b\not=0$$ by Serre duality corresponds to $$h^1(E^\vee(-f))=h^1(E^\vee (H+(c-2)f)\otimes\Oo_X(-H-(c-1)f))=h^1(E'(-H-(c-1)f)\not=0.$$ Let us consider for any integer $t$ the exact sequence $$0\to E'(-H-(t+2)f)\to E'(-H-(t+1)f)\to E'(-H-tf))\to 0.$$ If $t\geq 0$ the map
 $$H^1(E'(-H-(t+2)f))\to H^1(E'(-H-(t+1)f)$$ is injective. Since $h^1(E'(-H-(c-1)f))\not=0$, we have $h^1(E'(-H-(c-2)f))\not=0$ and by a recursive argument $h^1(E'(-H-tf))\not=0$ for $-(c-1)\leq t\leq 1$.
 In particular $h^1(E'(-H-f))\not=0$, so we may repeat the argument of the case $(\alpha)$ for the Ulrich bundle $E'$ and we obtain $E'\cong \Oo_X(H-f)$ (hence $E\cong \Oo_X((c-1)f)$) or $E'$ fits in the extension $$0 \to \sO_X(H-F)^{b} \to E' \to \sO_X((c-1)f)^{a} \to 0$$ that dualized and tensored by $\Oo_X(H+(c-2)f)$ becomes (\ref{Ul1}).
\end{proof}
\begin{remark}If $c=2$, $\dim (Ext^1(\sO_X((c-1)f),\sO_X(H-f)))=0$ so (\ref{Ul1}) splits.\\ 
If $c=3$, $\dim (Ext^1(\sO_X((c-1)f),\sO_X(H-f)))=1$ so from (\ref{Ul1}) we only obtain a unique rank two indecomposable Ulrich bundle.\\
If $c=4$, $\dim (Ext^1(\sO_X((c-1)f),\sO_X(H-f)))=2$ so from (\ref{Ul1}) we only obtain families of dimension at most one of indecomposable Ulrich bundles (see \cite{FM}).\\
If $c>4$, $\dim (Ext^1(\sO_X((c-1)f),\sO_X(H-f)))>2$ so from (\ref{Ul1}) we may obtain arbitrary large families of indecomposable Ulrich bundles (see \cite{FM}).\\
\end{remark}

\section{Logarithmic bundle on $S(a_0,a_1)$}
\begin{definition}
A divisor $D$ on a non-singular variety $X$ is said to have normal crossings if $\Oo_{D,x}$ is formally isomorphic to the quotient of $\Oo_{X,x}$ by an ideal generated by $t_1,\ldots,t_k$, where $t_1,\ldots,t_k$ is a subset of the set of local parameters in $\Oo_{X,x}$ for all $x \in D$. $D$ is also said to have \em simple normal crossings\em\  if it is the union of
smooth divisors $D_i, i = 1,\ldots,m$, which intersect transversely at each point.
\end{definition} 
\begin{definition}
 An arrangement on $X$ is defined to be a set $D = \{D_1,\ldots , D_m\}$ of smooth irreducible divisors of $X$ such that $D_i 
\neq D_j$ for $i \neq j$. To an arrangement $D$ on $X$, we can associate the
\em logarithmic sheaf\em\ $\Omega^1
_X(\log D)$, the sheaf of differential $1$-forms with logarithmic poles along $D$.
\end{definition}
If $D$ has simple normal crossings, its logarithmic sheaf is known to be locally free and so it can
be called to be the logarithmic bundle. It admits the residue exact sequence
\begin{equation}
0 \rightarrow \Omega^1_
X \rightarrow \Omega^1_
X(\log D)
\stackrel{res}{\rightarrow}
\bigoplus {\epsilon_{i^*}}  \Oo_{D_i}\rightarrow 0,
\end{equation}
From now on, let $X$ be $S(a_0,a_1)$ and let $e=a_1-a_0$. Let us consider the lines $L_i\in |\sO_X(f)|$. Recall (see \cite{ha2} II 8.11.) that the cotangent bundle of $X$ is given in 
$$0\to \Oo_X(-2f) \to \Omega_X^1 \to \Oo_{X}(-2H+cf) \to 0,$$ and this extension splits only if $e=0$:
\begin{proposition}\label{line}
Let $D=\{L_1, \cdots, L_a\}$ be an arrangement of $a$ lines on $X$ with $L_i \in |\Oo_X(f)|$. Then we have
$$\Omega_X^1 (\log D) \cong \Oo_X((a-2)f)\oplus \Oo_X(-2H+cf),$$ if $a\geq e+1.$
\end{proposition}
\begin{proof}
Let $D=\{L_1\}$. We apply the covariant functor $Hom(\Oo_{L_1},-)$ to vertical column of the diagram 
$$\begin{array}{ccccc}& 0 & & & \\ &\downarrow && & \\ & \Oo_X(-2f) &  &&\\ &\downarrow &&& \\
0\rightarrow& \Omega^1_X &\rightarrow &\Omega^1_X(\log D)&\rightarrow \Oo_{L_1}\to 0
 \\
&\downarrow && & \\
 & \Oo_X(-2H+cf)&&
  \\
&\downarrow && & \\ & 0 & & & 
 \end{array}$$
As the dimension of $Ext^1 (\Oo_{L_1}, \Oo_X(-2H+cf))$ is $h^1(\Oo_{L_1}\otimes\Oo_X(-2f))=h^1(\Oo_{\mathbb P^1})=0$ and the dimension of $Ext^0 (\Oo_{L_1}, \Oo_X(-2H+cf))$ is $h^2(\Oo_{L_1}\otimes\Oo_X(-2f))=h^2(\Oo_{\mathbb P^1})=0$, we get $$Ext^1(\Oo_{L_1},\Omega^1_X)\cong Ext^1 (\Oo_{L_1}, \Oo_X(-2f))$$ and their dimension is $h^1(\Oo_{L_1}\otimes\Oo_X(-2H+cf))=h^1(\Oo_{L_1}(-2))=h^1(\Oo_{\mathbb P^1}(-2))=1$. So we get the unique extension (note that $\Oo_{L_1}\otimes\Oo_X(-f)\cong\Oo_{L_1}$) $0\to \Oo_X(-2f) \to \Oo_X(-f) \to \Oo_{L_1} \to 0$ to close the following diagram
 $$\begin{array}{ccccc}& 0 & &0 & \\ &\downarrow &&\downarrow  & \\ 0\to& \Oo_X(-2f) & \to &\Oo_X(-f) &\to \Oo_{L_1} \to 0.\\ &\downarrow &&\downarrow &\| \\
 0 \rightarrow& \Omega^1_X &\rightarrow &\Omega^1_X(\log D)&\rightarrow \Oo_{L_1}\to 0
 \\
&\downarrow &&\downarrow & \\
 & \Oo(-2H+cf)&=&\Oo(-2H+cf)
  \\
&\downarrow &&\downarrow & \\ & 0 & &0 & 
 \end{array}$$
 \\ Now, let $a\geq 2$, $D={L_1,\ldots,L_a}$, $D'={L_1,\ldots,L_{a-1}}$ and the assertion true for $a-1$ to argue by induction. Similarly, we get the diagram
  $$\begin{array}{ccccc}
  & 0 & &0 & \\ &\downarrow &&\downarrow  & \\ 0\to& \Oo_X((a-3)f) & \to &\Oo_X((a-2)f) &\to \Oo_{L_1} \to 0.\\ &\downarrow &&\downarrow &\| \\
 0 \rightarrow& \Omega^1_X(\log ((a-1)f)) &\rightarrow &\Omega^1_X(\log (af))&\rightarrow \Oo_{L_1}\to 0\\
&\downarrow & &\downarrow & \\
 & \Oo_X(-2H+cf)&=&\Oo_X(-2H+cf) & \\
 &\downarrow &&\downarrow & \\
 & 0 & &0 & 
 \end{array}$$
 in fact, the dimension of $Ext^1 (\Oo_{L_1}, \Oo_X((a-3)f)))$ is $h^1(\Oo_{L_1}\otimes\Oo_X(-2H+(c-e)f)))=h^1(\Oo_{L_1}(-2)=h^1(\Oo_{\mathbb P^1}(-2)))=1$ and the dimension of $Ext^1 (\Oo_{L_1}, \Oo_X(-2H+cf)$ is $h^1(\Oo_{L_1}\otimes\Oo_X(-2f))=0$.\\ 
 In order to have a splitting sequence in the second column, it is enough that $Ext^1(\Oo_X(-2H+cf),\Oo_X((a-2)f)=0$.  We have that the dimension of $Ext^1(\Oo_X(-2H+cf),\Oo_X((a-2)f))$ is $h^1(\Oo_X(2H+(a-2-c)f))$ which is zero if $ h^1(\Oo_{\mathbb P^1} (2a_0+a-2-c))=0$. This implies $2a_0+a-c\geq 1$, hence $a\geq e+1.$
\end{proof}
\begin{proposition}\label{lineC}
Let $D=\{L_1, \cdots, L_a, C_1\}$ be an arrangement of $a\geq e+1$ lines and one rational curve on $X$ with $L_i \in |\Oo_X(f)|$ and $C_1\in |\Oo_X(H-a_1f)|$. Then we have
$$\Omega_X^1 (\log D) \cong \Oo_X((a-2)f)\oplus \Oo_X(-H+a_0f). $$
\end{proposition}
\begin{proof}
Let us first consider the case of $D=\{L_1, \cdots, L_{e+1}, C_1\}$. Let $D'=\{L_1, \cdots, L_{e+1}\}$. Then we have the sequence
$$0\to \Omega_X^1 (\log D') \to \Omega_X^1 (\log D) \to \Oo_{C_1} \to 0.$$
Thanks to Proposition \ref{line}
$$\Omega_X^1 (\log D') \cong \Oo_X((e-1)f)\oplus \Oo_X(-2H+cf).$$ 
Note that the dimension of $Ext^1 (\Oo_{C_1}, \Oo_X((a-2)f))$ is $h^1(\Oo_{C_1}\otimes\Oo_X(-2H+(c-a)f))=h^1(\Oo_{\mathbb P^1}(-2c+(c-a+2a_1))=h^1(\Oo_{\PP^1}(-c+2a_1-a))=h^1(\Oo_{\PP^1}(e-a))=0$ if $a\leq e+1$ and the dimension of $Ext^1 (\Oo_{C_1}, \Oo_X(-2H+cf)$ is $h^1(\Oo_{C_1}\otimes\Oo_X(-2f))=h^1(\Oo_{\mathbb P^1}(-2))=1$ and we get the unique extension
$$0\to \Oo_X(-2H+cf) \to \Oo_X(-H+(c-a_1)f) \to \Oo_{C_1} \to 0.$$
 Thus there exists a uniquely determined extension of $\Oo_{C_1}$ by $\Omega_X^1 (\log D')$ and it must be $\Oo_X((e-1)f)\oplus \Oo_X(-H+a_0f)$.\\

Now assume that the assertion is true for $a\geq e+1$ to use induction. For the case of $a+1$, if $D'=\{L_1, \cdots, L_a, C_1\}$  we have the sequence
$$0\to \Omega_X^1 (\log D') \to \Omega_X^1 (\log D) \to \Oo_{L_{a+1}} \to 0.$$
Thanks to the above argument
$$\Omega_X^1 (\log D') \cong \Oo_X((a-2)f)\oplus \Oo_X(-H+a_0f),$$
Note that the dimension of $Ext^1 (\Oo_{L_{a+1}}, \Oo_X((a-2)f))$ is $h^1(\Oo_{L_{a+1}}\otimes\Oo_X(-2H+(c-a)f))=h^1(\Oo_{\mathbb P^1}(-2)=1$ and the dimension of $Ext^1 (\Oo_{L_{a+1}}, \Oo_X(-H+a_of))$ is $h^1(\Oo_{L_{a+1}}\otimes\Oo_X(-H+(a_1-2)f))=h^1(\Oo_{\mathbb P^1}(-1))=0$ and we get the unique extension
$$0\to \Oo_X((a-2)f) \to \Oo_X((a-1)f) \to \Oo_{L_{a+1}} \to 0.$$
 Thus there exists a uniquely determined extension of $\Oo_{L_{a+1}}$ by $\Omega_X^1 (\log D')$ and it must be $\Oo_X((a-1)f)\oplus \Oo_X(-H+a_0f)$.

\end{proof}

If $e>0$, $h^0(\Oo_X(H-a_1f))=1$ and we cannot consider an arrangement with more than a curve $C_j\in |\Oo_X(H-a_1f)|$. When $e=0$, $h^0(\Oo_X(H-a_1f))=2$ and a curve $C\in |\Oo_X(H-a_1f)|$ is rational of degree $a_0$.  Moreover $(H-a_1f)^2=0$ so we have a one dimensional family of disjoint lines in $|\Oo_X(f)|$ and a one dimensional family of  disjoint rational curve of degree $a_0$. For this reason we consider now in more details the case $e=0$.

\begin{theorem}\label{lineCb}
Let $e=0.$ Let $D=\{L_1, \cdots, L_a, C_1, \cdots, C_b\}$ be an arrangement of $a\geq 0$ lines and $b\geq 0$ rational curves on $X$ with $L_i \in |\Oo_X(f)|$ and $C_j\in |\Oo_X(H-a_1f)|$. Then we have
$$\Omega_X^1 (\log D) \cong \Oo_X((a-2)f)\oplus \Oo_X((b-2)H+(c-ba_1)f).$$

\end{theorem}
\begin{proof}
 Let us first consider the case of $D=\{ C_1, C_2\}$. Let $D'=\{C_1\}$. Then we have the sequence
$$0\to \Omega_X^1 (\log D') \to \Omega_X^1 (\log D) \to \Oo_{C_2} \to 0.$$
Thanks to Proposition \ref{lineC}
$$\Omega_X^1 (\log D') \cong \Oo_X(-2f)\oplus \Oo_X(-H+a_0f).$$ 
Note that the dimension of $Ext^1 (\Oo_{C_2}, \Oo_X((-2f))$ is $h^1(\Oo_{C_2}\otimes\Oo_X(-2H+cf))=h^1(\Oo_{\PP^1}(-2c+c+2a_1))=h^1(\Oo_{C\PP^1}(-c+2a_1))=0$ and the dimension of $Ext^1 (\Oo_{C_2}, \Oo_X(-H+a_0f)$ is $h^1(\Oo_{C_2}\otimes\Oo_X(-H+(c-2-a_0)f))=h^1(\Oo_{\mathbb P^1}(-c+c-2+a_1-a_0))=1$ and we get the unique extension
$$0\to \Oo_X(-H+a_1f) \to \Oo_X \to \Oo_{C_2} \to 0.$$
 Thus there exists a uniquely determined extension of $\Oo_{C_2}$ by $\Omega_X^1 (\log D')$ and it must be $\Oo_X((-2f)\oplus \Oo_X$.\\


  Now assume that the assertion is true for $(0,b)$ with $b\geq 2$ to use induction. For the case of $(0,b+1)$ $D=\{ C_1, \dots ,C_{b+1}\}$. Let $D'=\{C_1, \cdots, C_b\}$, then we have the sequence
$$0\to \Omega_X^1 (\log D') \to \Omega_X^1 (\log D) \to \Oo_{C_{b+1}} \to 0.$$
Thanks to above argument and the inductive hypothesis, 
$$\Omega_X^1 (\log D') \cong \Oo_X(-2f)\oplus \Oo_X((b-2)H+(c-ba_1)f).$$ Let us tensor the above sequence by $\Oo_X(2f)$ and we obtain
$$0\to \Oo_X\oplus \Oo_X((b-2)H+(c-ba_1+2)f)\to \Omega_X^1 (\log D)\otimes \Oo_X(2f) \to \Oo_{C_{b+1}}\otimes\Oo_X(2f) \to 0.$$
We call $E=\Omega_X^1 (\log D)\otimes \Oo_X(2f)$ and we want to show that $E$ is regular. Notice that  $\Oo_X\oplus \Oo_X((b-2)H+(c-ba_1+2)f)$ is regular. So $h^1(E(-f))=h^1(\Oo_{C_{b+1}}\otimes\Oo_X(-f))=h^1(\Oo_{\mathbb P^1}(-1))=0$. Moreover $h^1(E(-H+(c-1)f))=h^1(\Oo_{C_{b+1}}\otimes\Oo_X(-H+(c+1)f))=h^1(\Oo_{\mathbb P^1}(-c+a_1+c+1))=h^1(\Oo_{\mathbb P^1}(a_1+1))=0$ and $h^2(E(-H+(c-2)f))=h^2(\Oo_{C_{b+1}}\otimes\Oo_X(-H+cf))=0$. Thus we have that $E$ is regular and, since $h^1(\Oo_{C_{b+1}}\otimes\Oo_X(-2H+cf))=h^1(\Oo_{\mathbb P^1}(-2c+2a_1+c))=0$, we get $h^2(E(-2H+(c-2)f))=h^0(E))\geq 1$ so, by Corollary \ref{cor1}, we can conclude that $\Oo_X$ is a direct summand of $E$. Hence, if $E$ is a vector bundle, $E\cong \Oo_X\oplus\Oo_X((b-1)H+(c-(b+1)a_1+2)f)$ and $$\Omega_X^1 (\log D) \cong \Oo_X(-2f)\oplus \Oo_X((b-1)H+(c-(b+1)a_1)f).$$ 

Finally  let us deal with the case when $a$ and $b$ is at least $1$ and $b$ at least $2$. The logarithmic bundle $\Omega_X^1 (\log D)$ is an extension of $(\oplus \Oo_{L_i})\oplus (\oplus \Oo_{C_j})$ by $\Oo_X(-2f)\oplus \Oo_X(-H+a_0f)$. Note that we have
$$Ext^1 (\oplus \Oo_{L_i}, \Oo_X(-H+a_0f))=0$$ by Proposition \ref{lineC} and $$Ext^1 (\oplus \Oo_{C_j}, \Oo_X(-2f))=0$$
by the above argument.\\
Thus $\Omega_X^1 (\log D)$ corresponds to an element $\epsilon$;
$$\epsilon \in Ext^1 (\oplus \Oo_{L_i}, \Oo_X(-2f))\oplus Ext^1 (\oplus \Oo_{C_j}, \Oo_X(-H+a_0f)).$$
From the argument in Proposition \ref{lineC}, we observe that the first factor of $\epsilon$ with $Ext^1(\oplus \Oo_{L_i}, \Oo_X(-H+a_0f))=0$ generates $\Oo_X((a-2)f)\oplus \Oo_X(-H+a_0f)$ and similarly,  by the argument for the case $(0,b)$ the second factor generates $\Oo_X(-2f)\oplus \Oo_X((b-2)H+(c-ba_1)f)$.
Thus $\epsilon$ corresponds to the bundle $\Oo_X((a-2)f)\oplus \Oo_X((b-2)H+(c-ba_1)f)$.

\end{proof}
\begin{remark}
    When $c=2$ the above Theorem coincides with \cite{bhm} Proposition 6.3.
\end{remark}
We are able finally to classify regular ACM logarithmic bundles:

\begin{corollary}
Let $e=0$ and $c>2$. Let $D$ be an arrangement of smooth curves on $X$ with simple normal crossings. If $\Omega_X^1 (\log D)$ is a regular ACM bundle, then $D$ consists of $a$ lines in $|\Oo_X(f)|$ with $2\leq a \leq c+1$ and $2$ rational curves in $|\Oo_X(H-a_1f)|$. In particular we have that $\Omega_X^1 (\log D)$ has always regularity $0$:
$$\Omega_X^1 (\log D) \cong \Oo_X((a-2)f)\oplus \Oo_X.$$ 
\end{corollary}
\begin{proof}
If $D:=\{D_1, \cdots, D_m\}$ consists of $m$ smooth curves, then it admits the sequence
\begin{equation}\label{seq4}
0\to \Oo_X(-2f)\oplus \Oo_X(-2H+cf) \to \Omega_X^1 (\log D) \to \oplus_{i=1}^m \Oo_{D_i} \to 0.
\end{equation}
If $E=\Omega_X^1(\log D)$ is regular, then  we have $h^1(E(-f))=0$. From the sequence (\ref{seq4}) twisted by $\Oo_X(-f)$ and the fact that $h^2(\Oo_X(-3f))=h^2(\Oo_X(-2H+(c-1)f))=0$, we deduce that $h^1(\Oo_{D_i}\otimes \Oo_X(-f))=0$ for any $i=1,\dots m$. Let $D_i\in |\Oo_X(s_iH+t_if)|$ with $s_i\geq 0$ and $t_i\geq -a_1$, we get $h^1(\Oo_{D_i}\otimes \Oo_X(-f))=h^1(\Oo_{\mathbb P^1}(-s_i))=0$, so $s_i=0$ or $s_i=1$.\\ 
Moreover by Lemma \ref{l5b} $H^1(E((a-1)H+bf))=H^1(E(aH+(b-1)f))=0$ for any $a,b\geq 0$, in particular we have $h^1(E(-H+(a_1-1)f))=0$. From the sequence (\ref{seq4}) twisted by $\Oo_X(-H+(a_1-1)f)$ and the fact that $h^2(\Oo_X(-H+(a_1-3)f))=0$ and $h^2(\Oo_X(-3H+(c+a_1-1)f))=h^0(\Oo_X((H+(-1-a_1)f))=0$, we deduce that $h^1(\Oo_{D_i}\otimes \Oo_X(-H+(a_1-1)f))=0$ for any $i=1,\dots m$. We get $h^1(\Oo_{D_i}\otimes \Oo_X(-H+(a_1-1)f))=h^1(\Oo_{\mathbb P^1}(-cs_i+(a_1-1)s_i-t_i)=0$. If $s_i=0$ we obtain $t_i=1$ so $D_i\in |\Oo_X(f)|$. If $s_i=1$ we must have $-c+a_1-1-t_i\geq -1$, hence $t_i\leq -a_0$. Since $a_0=a_1$ and $t_i\geq -a_1$ we may conclude that $t_i=-a_1$. So we have only two cases: $D_i\in|\Oo_X(f)|$ or $D_i\in|\Oo_X(H-a_1f)|$. By Proposition \ref{lineCb} $$\Omega_X^1 (\log D) \cong \Oo_X((a-2)f)\oplus \Oo_X((b-2)H+(c-ba_1)f).$$ We recall that $\Oo_X(sH+tf)$ si ACM if and only if $-1\leq t\leq c-1$ so we mast have $2\leq a \leq c+1$ and (since $c>2$) $b=2$.
\end{proof}


\bibliographystyle{amsplain}

\end{document}